\documentclass[12pt]{amsart}
\usepackage[english]{babel}
\usepackage{amsmath}
\usepackage{amsthm}
\usepackage{amsfonts,dsfont}
\usepackage{amssymb}
\usepackage{mathrsfs}
\usepackage{hyperref}

\usepackage{tikz}
\usetikzlibrary{arrows}
\usetikzlibrary{arrows.meta}
\usetikzlibrary{decorations.markings}
\usetikzlibrary{shapes.geometric, arrows}
\usepackage[margin=1in]{geometry}
\numberwithin{equation}{section}
\allowdisplaybreaks

\def\emptyset{\varnothing}
\def\ti{\to\infty}

\def\1{{\mathchoice {1\mskip-4mu\mathrm l}      % Blackboard bold 1
{1\mskip-4mu\mathrm l}
{1\mskip-4.5mu\mathrm l} {1\mskip-5mu\mathrm l}}}

\def\comment#1{}

\newcommand{\eps}{\varepsilon}
\renewcommand{\epsilon}{\varepsilon}
\newcommand{\Fcal}   {{\mathcal F }}

\def\ignore#1{}

\newcommand{\R}     {\mathbb{R}}
\newcommand{\Z}     {\mathbb{Z}}
\newcommand{\N}     {\mathbb{N}}
\renewcommand{\P}   {\mathbb{P}}
\newcommand{\E}     {\mathbb{E}}
\def\bbP{\mathbb{P}}
\def\bbE{\mathbb{E}}

\newtheorem{assumption}{Assumption}
\newtheorem{thm}{Theorem}[section]
\newtheorem{lem}[thm]{Lemma}

\newtheorem{prop}[thm]{Proposition}
\newtheorem{rmk}[thm]{Remark}
\newtheorem{Ex}[thm]{Example}

\title[Long range 1-cookie random walk with positive speed]{Long range one-cookie random walk with positive speed
 }
\date{}

\author{Andrea Collevecchio}
\address[\sc A. Collevecchio]{School of Mathematics, Monash University, Victoria 3800, Australia}
\email{andrea.collevecchio@monash.edu}
\author{Kais Hamza}
\address[\sc K. Hamza]{School of Mathematics, Monash University, Victoria 3800, Australia}
\email{kais.hamza@monash.edu}

\author{Tuan-Minh Nguyen}
\address[\sc T.M. Nguyen]{School of Mathematics, Monash University, Victoria 3800, Australia}
\email{tuanminh.nguyen@monash.edu}


\begin{document}
\begin{abstract}
We  study one-dimensional excited random walks with non-nearest neighbour jumps. 
When the process is at a vertex that has not been visited before, its next transition has a positive drift to the right, possibly with long jumps. Whenever the process visits a vertex that has already been visited in the past, its next transition is the one of a simple symmetric random walk. 
We give a sufficient condition for the process to have positive speed.
\end{abstract}
\maketitle

\section{Introduction}
Cookie random walks (also known as excited random walks) provide a model for random motion with long memory where the transition out of a vertex depends on the number of visits made in the past to that vertex. This phenomenon is illustrated by initially placing a number of cookies on each vertex and consuming one cookie at each visit {to} the vertex. While the number of cookies on a particular vertex is non-zero, the random motion out of the vertex is in excited mode, that is the walk behaves locally like a biased random walk and its drift depends on the number of remaining cookies at the present vertex.
Once the vertex runs out of cookies, the random motion reverts to that of an unbiased (or symmetric) random walk. Such processes are clearly non-Markovian and their transition laws depend on the local time of the walk at its current location.

This model was introduced by Benjamini and Wilson \cite{BW2003} and was later generalized by Zerner~\cite{Z2005}. In recent years, significant attention was afforded to the study of one-dimensional nearest-neighbour excited random walks resulting in remarkable phase transition results for the asymptotic behavior of the model, including criteria for recurrence/transience \cite{ABO, Z2005} and non-ballisticity/ballisticity \cite{BasS1, BasS2}, as well as characterization of the limit distribution in certain specific regimes -- see \cite{DK2012, KM2011, KZ2013, P2012} for more details.

In the present paper, we study a model of one-cookie random walks with long jumps to the right and prove that non-nearest neighbour one-cookie random walks may have positive speed. These processes are defined as follows. Fix a distribution $q$ on the set of integers  $\{-1, 0, 1, \ldots, L\}$.  When the process visits a vertex for the first time, its next jump has distribution $q$. When the walk visits a vertex that was previously visited, it behaves like a nearest-neighbour simple symmetric random walk. It's worth mentioning that our model is a special case of the multi-excited random walks with non-nearest neighbour jumps considered by Davis and Peterson~\cite{Dav} in which the support of the jumps can be an unbounded subset of $\Z$. These authors proved an explicit criterion for recurrence/transience for this general model. They also conjectured the existence of speed and a criterion for ballisticity. We prove the first conjecture in the case of bounded-from-above jumps and the second one in a single cookie environment under a stricter criterion on the tail probabilities of $q$.

\section{Main results}

For any $m,n \in \Z$, $m\le n$, set $[m,n]_\Z:= \{m, \ldots, n\}$.
Fix $C \in \N$ and probability measures $(q_i)_{i \in [1,C]_\Z}$ on $\Z$. Denote by $\langle q_i \rangle = \sum_{k\in\Z} kq_i(k)$  the mean value of $q_i$. Throughout this paper, we require the following condition:
\begin{assumption}\label{assump:0}
Assume that for each $i\in[1,C]_{\Z}$, 
\begin{itemize}
\item $\langle q_i \rangle\ge 0$, \item $q_i$ has the same support $\Lambda:=\{k\in \Z : q_i(k)>0\}$ {which contains at least two elements}. 
\end{itemize}
\end{assumption}
We define the long range $C$-cookie random walk $Y$ as follows. Assign to each vertex in $\Z$ exactly $C$ cookies and set $Y_0 =0$. Each time the process visits a vertex that has exactly $i\in [1,C]_\Z$ cookies, it eats one of them, thus reducing the number of cookies at the current vertex by 1, and uses $q_i$ to determine its next transition. If the vertex has no cookies, then the walk moves to one of its two nearest neighbours with equal probability; that is it uses, $q_0$,  the symmetric measure on $\{-1,1\}$ to decide on its next step. More formally,  the transition distribution of $Y$ is given by
$$\P\Big(Y_{t+1}=Y_{t}+k\ |\ \mathcal{F}_t\Big)=q_{\mathcal{C}(Y_t,t)}(k), \text{ for } t\in \Z_+ \text{ and } k\in \Lambda,$$
where $\mathcal{F}_t$ is the $\sigma$-field generated by $(Y_s)_{0\le s\le t}$, and $\mathcal{C}(j,t)=\max\Big(C-\sum_{s=0}^{t-1}\1_{\{Y_s=j\}}, 0\Big)$ denotes the number of remaining cookies at vertex $j$ right before time $t$.
The process thus defined is a long-range $C$-cookie random walk on $\Z$.

This model is a special case of the multi-excited random walks with non-nearest neighbour steps considered by Davis and Peterson in \cite{Dav}. In their model, the support $\Lambda$ can be an arbitrary subset of $\Z$ (possibly unbounded). The authors in \cite{Dav} investigated the recurrence of such processes and established a phase transition driven by the key quantity, called \textit{the expected total drift},
$$\delta := \sum_{i=1}^C \langle q_i \rangle.$$
In particular, Theorem 1.6 in \cite{Dav} implies that the process $Y$ is recurrent, i.e. visits each vertex infinitely often,  if and only if $0\le\delta \le 1$. We define $\lim_{t\to\infty}Y_t/t$ to be the \textit{speed} of $Y$, when the limit exists.

The nearest-neighbour case, i.e. $\Lambda=\{-1,1\}$, has been extensively studied. When the walk is nearest-neighbour and $C=1$, the process is  always recurrent, as pointed out by Benjamini and Wilson in \cite{BW2003}. When $C$ is an arbitrary positive integer, Basdevant and Singh (see Theorem  1.1 in \cite{BasS1}) showed that the nearest-neighbour process has positive speed if and only if $\delta>2$. Since in this case $\langle q_i \rangle<1$, in order for $Y$ to have a positive speed, $C$ must be at least 3. 

In the present paper, we focus on the case of 1-cookie random walks skip-free to the left with bounded jumps to the right, i.e. the following assumption holds true
\begin{assumption}\label{asm:1}
$C=1$, $\inf(\Lambda)=-1$ and $L:=\sup(\Lambda)\in[1,\infty)$.
\end{assumption}
 More specifically, we provide a partial answer to the Davis and Peterson conjecture (see Conjecture 1.8 in \cite{Dav}), that $Y$ has positive speed if and only if $\delta>2$. In fact we show that positive speed results under a stricter requirement \eqref{eq:epsell} on the tail probabilities $$Q(j) = \sum_{k=j}^L q(k),$$ where we simply write $q$ for $q_1$.

\begin{thm}\label{th:mainth}
Consider a 1-cookie random walk $Y=(Y_t)_t$ with bounded jumps to the right. Under {Assumptions  \ref{assump:0} and \ref{asm:1}, the speed of $Y$ exists and is} positive 
if there exists a pair of integers  $c, \ell$ such that $c \ge 3$, $\ell \ge 3c$ and
\begin{equation}\label{eq:epsell}
 2 \left(1-\frac {c-1}{\ell}\right)Q(\ell+c-1) - 1 > \frac 2{c}.
\end{equation}
\end{thm}
\begin{Ex}
Suppose $q(\ell+c-1) = 1 -\eps$ and $q(-1) = \eps$, then condition \eqref{eq:epsell} implies that $Y$ has positive speed for $$\eps<\frac{1-2(c-1)/\ell-2/c}{2(1-(c-1)/\ell)}.$$ In particular, \eqref{eq:epsell} holds with the choice $\ell=13$, $c = 3$ (and therefore $L=15$) and $\eps<1/66$.
\end{Ex}

{We generate the process $Y$ using a family of uniform random variables as follows.  Let $(\zeta_t)_{t\in \N}$ be independent uniform [0,1] random variables. Set $\widehat{Y}_0=0$ and for each $t\ge 0$ we define
\begin{align}\label{eq:constr} \widehat Y_{t+1}= \widehat Y_{t}+\left\{ \begin{array}{ll}\sum_{j=-1}^L j \1_{\{Q(j)>\zeta_{t+1}> Q(j+1)\}}&\text{if } \sum_{j=0}^{t-1}\1_{\{  \widehat Y_{j}= \widehat Y_t\}}=0,\\
-\1_{\{\zeta_{t+1}>1/2\}}+\1_{\{\zeta_{t+1}<1/2\}}& \text{otherwise}.
\end{array}\right.\end{align}
It is clear that the process $\widehat Y=(\widehat Y_t)_t$ defined above is a 1-cookie random walk satisfying Assumptions~\ref{assump:0} and~\ref{asm:1}.}
\begin{rmk}\label{rmk1}\
\begin{itemize}
    \item[i.] \eqref{eq:epsell} is a stricter requirement than Conjecture 1.8 in \cite{Dav}.
Indeed, \eqref{eq:epsell} implies that $\delta>2$ as can be seen from the following inequalities:
$$
\begin{aligned}
2 &< 2c\left(1-\frac {c-1}{\ell}\right)Q(\ell+c-1) - c  < 2 c Q(\ell) - 1 = (2c-1)Q(\ell) - (1- Q(\ell))\\
&< \ell Q(\ell) - (1- Q(\ell)) \le \sum_{k=\ell}^Lkq(k) + \sum_{k=-1}^{\ell-1}kq(k) = \langle q\rangle.
\end{aligned}
$$
\item[ii.] We also note that \eqref{eq:epsell} implies that $Q(1)>1/2$ or equivalently that $q(-1)+q(0)<1/2$. {In this case, we deduce from \eqref{eq:constr} that $\widehat Y_{t+1}- \widehat Y_t\ge -\1_{\{\zeta_{t+1}>1/2\}}+\1_{\{\zeta_{t+1}<1/2\}}$. Hence, for each $t_0\in \Z_+$, the process $(Y_t)_{t\ge t_0}$ stochastically dominates a simple symmetric random walk started at $Y_{t_0}$.}
\end{itemize}

\end{rmk}

\section{Existence of the speed in a general model}
The following is a very slight adaptation of an idea by Zerner (see Lemma 3.2.5 in \cite{Z2004}). In particular, Proposition~\ref{prop:exs} covers the cases of transient multi-excited random walks with jumps unbounded from below but bounded from above.
\begin{prop}\label{prop:exs}
Consider a $C$-cookie random walk $Y$ {satisfying Assumptions~\ref{assump:0}}, where the jumps are bounded to the right, i.e. $L=\sup(\Lambda)\in [1,\infty)$. Assume that $\delta>1$. Then, the speed of the process $Y$ exists, i.e.
\begin{equation}
\lim_{t \ti} \frac{Y_t} t \in [0,\infty) \quad \text{(a.s.).}
\end{equation}
\end{prop}
\begin{proof}
Define $\tau_0=0$ and
$$ \tau_{n+1} = \inf\{t \ge \tau_n \colon Y_{t+1}- Y_{t} = L \mbox{ and } Y_s < Y_t \mbox{ and }Y_{t+1}\le Y_u \mbox{ for all } s<t<u\}.$$
The r.v. $\tau_1$ is the first time the walk reaches a new maximum, say $j$, has a forward jump of length $L$ and then never jumps back to a site smaller than $j+L$. The $\tau_n$'s define the consecutive such times and represent a sequence of cut-times of $Y$. 

{Since $\delta>1$, it follows from Theorem 1.6 in \cite{Dav} that the process $Y$ is a.s. transient to the right, i.e. $\lim_{t\ti}Y_t=\infty$. For $j\in \Z^+$ set $A_j := [jL,(j+1)L-1]_\Z$. Note that $Y$ must eventually get larger than the largest vertex in $A_j$, $(j+1)L-1$, and since $|Y_{t+1}-Y_t|\leq L$, $Y$ cannot skip over $A_j$, the length of which is $L$. In other words, $A_j$ will be reached from the left in a finite time, almost surely. Let $S^{(j)}=\inf\{t\geq0 \colon Y_t\in A_j\}$. Then on the event $\nabla_j :=\{\exists k\colon \tau_k = S^{(j)}\}$, $A_j$ is visited by the process $Y$ exactly once. Hence $\nabla_j$ occurs if and only if $Y_t \ge Y_{S^{(j)}+1},  \forall t \ge S^{(j)}+1$ and $Y_{S^{(j)}+1} - Y_{S^{(j)}}=L$. We thus have that
\begin{equation} \label{prob.nab}
\begin{aligned}
\P\big(\nabla_j\ \big|\ \mathcal{F}_{S^{(j)}}\big)& = \E\left[\P\big(\nabla_j\;\big|\; \Fcal_{S^{(j)}+1}\big)\ \big|\ \mathcal{F}_{S^{(j)}}\right]\\
&= \E\Big[\P\big(Y_t - Y_{S^{(j)}+1} \ge 0,  \forall t \ge S^{(j)}+1\;\big|\; \Fcal_{S^{(j)}+1}\big) \1_{\big\{Y_{S^{(j)}+1} - Y_{S^{(j)}} = L\big\}}\ \big|\ \mathcal{F}_{S^{(j)}}\Big]\\
&=\P(Y_t\ge 0, \ \forall t \ge 0) \E\Big[\1_{\big\{Y_{S^{(j)}+1} - Y_{S^{(j)}} = L\big\}}\ \big|\ \mathcal{F}_{S^{(j)}}\Big]\\
&= \alpha q_1(L),
\end{aligned}
\end{equation}
where $\alpha:=\P(Y_t\ge 0, \ \forall t \ge 0)$. By Lemma 5.6 in \cite{Dav}, we conclude that $\alpha>0$. The third step in \eqref{prob.nab} is justified by the fact that by time $S^{(j)}+1$ the walk has not visited any vertex to the right of $Y_{S^{(j)}+1}$. Hence, the probability that it never visits a vertex to the left of $Y_{S^{(j)}+1}$,  conditional on the past, equals $\P(Y_t\ge 0, \ \forall t \ge 0)$.}

{We next prove by induction that for each $k\ge 0$, $\tau_k$ is a.s. finite. Recall that $\tau_0=0$. Assuming that for some fixed $k$, $\tau_k$ is a.s. finite, we show that $\tau_{k+1}$ is also a.s. finite. Indeed, set $R^{(j)}:=\inf\{t\ge S^{(j)}+1: Y_t< Y_{S^{(j)}}+L\}$. Let $M$ is a fixed positive integer and $(\nu_n)_{n\ge0}$ be a sequence defined recursively by $\nu_0=M$ and $$\nu_{n+1}=\inf\big\{j\ge \nu_n+1: S^{(j)}\ge R^{(\nu_n)}\big\}.$$
Note that $\{\nu_{n+1}<\infty\}=\{\nu_n<\infty, R^{(\nu_n)}<\infty\}= \nabla_{\nu_n}^c\cap\{\nu_n<\infty\}$.  We thus have 
\begin{align*}\P(\nu_n<\infty)& =\prod_{i=1}^n\P\Big(\nu_i<\infty\ |\ \nu_{i-1}<\infty  \Big)
 = \prod_{i=1}^n\P\Big(\nabla_{\nu_{i-1}}^c\ |\ \nu_{i-1}<\infty  \Big)
 = (1-\alpha q_1(L))^n.
\end{align*}
The last identity follows from the fact that on $\{\nu_{i-1}<\infty\}$, we can reason as in \eqref{prob.nab} to get  $\P\big(\nabla_{\nu_{i-1}}^c\ |\   \mathcal{F}_{S^{(\nu_{i-1})}}\big)=1-\alpha q_1(L)$. Hence, by Borel-Cantelli lemma, there a.s. exists $N$ such that $\nu_{N-1}<\infty$, $\nu_N=R^{(\nu_{N-1})}=\infty$ and thus $\nabla_{\nu_{N-1}}$ occurs. Note that $\{\tau_k<M\}\cap\{\exists N: \nabla_{\nu_{N-1}} \text{ occurs}\}\subset\{\tau_{k+1}<\infty\}$. 
Since $M$ is an arbitrary positive integer and $\P\big(\tau_k<\infty\big)=\P\big(\exists N: \nabla_{\nu_{N-1}} \text{ occurs}\big)=1$, we obtain that a.s. $\tau_{k+1}<\infty$.  By the principle of induction, we conclude that for each $k$, $\tau_k$ is a.s. finite. Also note that $\tau_k\to\infty$ as $k\to\infty$.   
}

Let $J_k = Y_{\tau_k}$.
We have that for any $t$ there exists $k$ such that $\tau_k\le t < \tau_{k+1}$, and thus
$$ \frac{J_k}{\tau_{k+1}}  \le \frac{Y_t}t <  \frac{J_{k+1}}{\tau_k}. $$
Moreover,  $(J_{k+1} - J_k, \tau_{k+1} - \tau_k )_{k\ge 1}$ is a sequence of independent and identically distributed random variables independent of $(J_1,\tau_1)$. Hence,  if we prove that $\E[J_2-J_1]< \infty$, we can appeal to the strong law of large numbers to conclude that 
$$
\lim_{t \ti} \frac{Y_t} t = \frac{\E[J_2-J_1]}{\E[\tau_2-\tau_1]}.
$$
In order to prove that $\E[J_2-J_1]< \infty$ we reason as follows. Set $\Omega_j :=\{\exists k\colon J_k \in A_j\}$. {By virtue of \eqref{prob.nab}, we have
$$ \P(\Omega_j) \ge \P\big(\nabla_j\big)\ge \alpha q_1(L)>0.$$
On  the other hand,}
$$
\begin{aligned}
& \limsup_{j \ti} \bbP(\Omega_j)
= \limsup_{j \ti} \bbP(\exists k \colon J_k \in  A_j) = \limsup_{j \ti} \bbP(\exists k \ge 2 \colon J_k \in  A_j)  \\
&= \limsup_{j \ti}\sum_{u=jL}^{(j+1)L-1} \bbP({\exists k\ge 2} \colon J_k = u)\\
& = \limsup_{j \ti} \sum_{u=jL}^{(j+1)L-1}\sum_{v=1}^\infty\bbP(\exists k \ge 2 \colon J_k=u|J_1=v) \bbP(J_1=v)\\
&\le \sum_{v=1}^\infty\bbP(J_1=v) \limsup_{j \ti}\sum_{u=jL}^{(j+1)L-1}\bbP(\exists k \ge 2 \colon J_k - J_1 = u-v),
\end{aligned}
$$
where the last inequality comes from the reverse Fatou's Lemma.
As the sequence $(J_{k+1}- J_k)_{k\in \N}$  is composed of i.i.d. (and independent of $J_1$), the sequence $(J_k-J_1)_{k\in \N}$ are the arrival times in a renewal process. Using the Blackwell Renewal Theorem (see e.g. \cite{D2010}), we have that 
$$ 
\lim_{u \ti} \bbP(\exists k \ge 2 \colon J_k - J_1 = u-v)  = \frac 1{\bbE[J_2 - J_1]}.
$$
Hence,
$$ 0 < {q_1(L) \alpha}  \le  \sum_{v=1}^\infty \bbP(J_1=v)\frac {L}{\bbE[J_2 - J_1]} = \frac {L}{\bbE[J_2 - J_1]}.
$$ 

\end{proof}
\section{Arrow systems}\label{sec:ArrowSys}
The main tool in the proof of Theorem~\ref{th:mainth} is a coupling with a nearest-neighbour multi-excited random walk. In order to establish such a coupling, we use arrow systems, which can be roughly described as follows.

Let's consider a {nearest-neighbour walk} $X$ on $\Z$. To each site $j\in\mathbb Z$ we assign an infinite sequence $(\mathcal{E}(j,k))_{k\in \N} \in \{-1,+1\}^{\N}$
  which describes the directions of the transitions of $X$  at each visit {to} the site. We also call $+1$ a right-pointing arrow ($\rightarrow$) and $-1$ a left-pointing arrow ($\leftarrow$). More precisely, we initially set $\mathcal{E}(j,k)=+1$ for all $j\in \Z, k\in \N$ and update arrows at each step of the walk. If the $k$-th transition out of a vertex $j$ is to $j-1$, then we change $\mathcal{E}(j, k)=-1$. Likewise, if the $k$-th transition out of a vertex $j$ is to $j+1$, then we keep $\mathcal{E}(j, k)=+1$. The infinite array $\mathcal{E}$ is called the \textit{arrow system} associated with $X$ and $\mathcal{E}(j,.)$ is called the \textit{arrow stack} at site $j$.    

Conversely, for each arrow system $\mathcal{E} = \big(\mathcal{E}(j, k)\big)_{j \in \Z, k \in \N}$, we can construct a corresponding nearest-neighbour walk $X^{\mathcal{E}}$ starting at 0. This can be seen using the following strong (pointwise) construction. 

Let $m^{\mathcal{E}}_t(j) = {\rm card}\{s\le t\colon  X^{\mathcal{E}}_s =j \}$ and $M^{\mathcal{E}}_t = m^{\mathcal{E}}_t(X^{\mathcal{E}}_t)$. Define $X^{\mathcal{E}}$ recursively as
$$X^{\mathcal{E}}_{t+1} = X^{\mathcal{E}}_t + \mathcal{E}(X^{\mathcal{E}}_t,M^{\mathcal{E}}_t).$$

We say that $\mathcal{E}$ dominates $\mathcal{E}'$, denoted by $\mathcal{E} \succeq\mathcal{E}'$, if
\begin{equation}\label{eq:ord1}
\sum_{k=1}^n \mathcal{E}(j, k) \ge \sum_{k=1}^n \mathcal{E}'(j, k) \qquad \mbox{for all } j \in \Z  \mbox{ and } n \in \N.
\end{equation}
 Holmes and Salisbury~\cite{MS2012} proved that this dominance extends to the nearest neighbour walks they define. 
\begin{thm}[Theorem 1.3.(iii) in \cite{MS2012}]\label{th:order} Consider two arrow systems $\mathcal{E}$ and $\mathcal{E}'$ such that
$\mathcal{E} \succeq\mathcal{E}'$ and $\limsup_{t\to\infty}X^{\mathcal{E}}_t\ge x$, for some $x \in \Z$.
Then,
$$\limsup_{t\ti}\frac{X^{\mathcal{E}}_t}t\geq
\limsup_{t\ti}\frac{X^{\mathcal{E}'}_t}t.$$
\end{thm}

\section{Arrow systems coupled with $Y$}
As mentioned in the previous section, the main idea of the paper is a coupling of a long-range once-excited random walk with a nearest-neighbour multi-excited random walk. Roughly speaking, we achieve this by aggregating vertices into ``mega vertices''. The latter are made up of multiple vertices in the same interval {of size $c$}. This allows us to focus on movements between mega vertices while managing transitions within them. In order to create a coupling with a $c$-cookie random walk on nearest neighbours, we first build an arrow system linked to the mega-vertices mentioned above. Each time the process enters one of the intervals  identified with some mega vertex, we distinguish two cases, depending on if that interval has cookies or not. 

{Let $Y=(Y_t)_t$ be a long-range one-cookie random walk satisfying Assumptions~\ref{assump:0} and~\ref{asm:1}. Suppose throughout this section that there exist fixed integers $c\ge 3$ and $\ell\ge 3c$ such that  \eqref{eq:epsell} is fulfilled.
Recall that the process $Y$ is transient to the right, i.e. $\lim_{t\to\infty}Y_t=\infty$.} We create an infinite array $\mathcal{E} \in \{-1, 1\}^{\Z\times \N}$ as follows. We initially set all elements of $\mathcal{E}$ equal to one, i.e.
$$\mathcal{E}(j, k) =1 \mbox{ for all $j \in \Z$ and $k \in \N.$}$$
Let $R_t = \{Y_s \colon s \le t\}$, the range of the process $Y$ up to time $t$ which is also the subset of $\Z$ containing all vertices with no cookie up to time $t$. We observe that the return to a vertex cannot occur through a (forward) jump over cookies, i.e. if $Y_t\in R_{t-1}$ and $Y_{t-1}<Y_t$, then $[Y_{t-1},Y_t]_\Z\subset R_{t-1}$.

For $j\in \Z^+$, we set 
$B_j := [\ell j,\ell j+c-1]_\Z$ which we also call a \textit{``mega vertex"}. Let $T^{(j)}=\inf\{t\geq0 \colon Y_t\in B_j\}$ be the first hitting time to the ``mega vertex" $B_j$.

For each discrete-time process $(Z_t)_t$, we denote by $\Delta Z_t:=Z_t-Z_{t-1}$ the backward difference of $(Z_t)_t$ at time $t$. Let $\mathcal{C}^{(j)}_t$ be the number of cookies remaining in $B_j$ {right before} time $t$, i.e. the sequence $(\mathcal{C}^{(j)}_t)_t$ is given by
$\mathcal{C}^{(j)}_0 = c\mbox{ and } \mathcal{C}^{(j)}_{t+1}  = c-{\rm card}(R_t\cap B_j).$
In particular $\mathcal{C}^{(j)}_{T^{(j)}} = c$, $\mathcal{C}^{(j)}_{T^{(j)}+1} = c-1$, and $\Delta\mathcal{C}^{(j)}_{t+1}=-1$ if and only if $Y_t\in (R_t\setminus R_{t-1})\cap B_j$.
Notice that $\mathcal{C}^{(j)}_{t+1}$ is $\mathcal{F}_t$-measurable.

We update {the arrow system $\mathcal{E}$ described above}, at random times $V^{(j)}_k$ (defined below) using the path of $Y$. More specifically the arrow $\mathcal{E}(j,k)$ is revised in a three-step triggering scheme. Formally, $\mathcal{E}(j, \cdot)$ is a process that gets updated at the times $\big(V^{(j)}_k\big)_k$.  For each fixed value of $j\in \Z$, the sequence $(V^{(j)}_k)_{k}$ is recursively defined as follows. Set $V^{(j)}_0=0$. Assume $V^{(j)}_{k-1}$ is defined. The triggers depend on the presence/absence of cookies in and around $B_j$.
\setlength{\leftmargini}{12pt}
\setlength{\leftmarginii}{12pt}
\begin{itemize}
\item  On $\big\{\mathcal{C}^{(j)}_{V^{(j)}_{k-1}}>0\big\}$,
\begin{enumerate}
\item[i.] the first trigger is realised when the process $Y$ enters $B_j$
\begin{equation*}\label{def.T1k}
T^{(j)}_{1,k}=\inf\{t\ge V^{(j)}_{k-1}:Y_t\in B_j\}.\end{equation*}
Notice that it is not necessary that the process lands directly on a vertex with a cookie. If it does land on a vertex with a cookie the first trigger coincides with the second trigger that is introduced below.
\item[ii.] The second trigger is activated when the process $Y$ hits a cookie in $B_j$ or it exits the interval $[\ell(j-1)+c,\ell j+c-1]_\Z$, i.e.
\begin{equation*}U^{(j)}_{1,k}=\inf\{t\ge T^{(j)}_{1,k} :Y_t\not\in[\ell(j-1)+c,\ell j+c-1]_\Z\mbox{ or }\Delta\mathcal{C}^{(j)}_{t+1 }=-1\},\end{equation*}
Note that if $Y$ hits a cookie at time $T_{1,k}^{(j)}$ then $U_{1,k}^{(j)}=T_{1,k}^{(j)}$.
\item[iii.] The third and final trigger occurs when the process $Y$ either jumps from a vertex in $B_j$ equipped with a cookie (i.e. a previously not visited vertex) or exits an interval that depends on whether or not $\mathcal{C}^{(j+1)}_{U^{(j)}_{1,k}}>0$. If $\mathcal{C}^{(j+1)}_{U^{(j)}_{1,k}}=0$ (resp. $\mathcal{C}^{(j+1)}_{U^{(j)}_{1,k}}>0$) then we require the process to exit the interval $[\ell(j-1)+c,\ell(j+1)+c-2]_\Z$ (resp. $[\ell(j-1)+c,\ell(j+1)-{1}]_\Z$). More formally,
$$V^{(j)}_{1,k}=\left\{\begin{array}{ll}
V^{(j)}_{10,k} & \mbox{if }\mathcal{C}^{(j+1)}_{U^{(j)}_{1,k}}=0,\\
V^{(j)}_{11,k} & \mbox{if }\mathcal{C}^{(j+1)}_{U^{(j)}_{1,k}}>0,
\end{array}\right.$$
\begin{align*}\text{ where } \quad 
V^{(j)}_{10,k} =\inf\{t\ge U^{(j)}_{1,k}:Y_t\not\in[\ell(j-1)+c,\ell(j+1)+c-2]_\Z\mbox{ or }\Delta{\mathcal{C}^{(j)}_{t}}=-1\}
\\
\text{ and } \quad V^{(j)}_{11,k}=\inf\{t\ge U^{(j)}_{1,k}:Y_t\not\in[\ell(j-1)+c,\ell(j+1)-{1}]_\Z\mbox{ or }\Delta {\mathcal{C}^{(j)}_{t}}=-1\},\end{align*}
where $\Delta\mathcal{C}^{(j)}_0=0$. Note that if $Y$ hits a cookie in $B_j$ at time $U_{1,k}^{(j)}$ then $V_{1,k}^{(j)}=U_{1,k}^{(j)}+1$. Note also that if the walk hits $\ell(j-1)+c-1$ (i.e. hits $B_{j-1}$) at time $U_{1,k}^{(j)}$ then $V_{1,k}^{(j)}=U_{1,k}^{(j)}$.

\end{enumerate}

\item On $\big\{\mathcal{C}^{(j)}_{V^{(j)}_{k-1}}=0\big\}$,
\begin{enumerate}
\item[i'.]  the first and second triggers are combined into one and are realised when the process $Y$ hits vertex $\ell j+c-1$, 
$$T^{(j)}_{0,k}=U^{(j)}_{0,k}=\inf\{t\ge V^{(j)}_{k-1}:Y_t=\ell j+c-1\}.$$
\item[ii'.] The third and final trigger occurs when the process $Y$ exits the interval $[\ell(j-1)+c,\ell(j+1)+c-2]_\Z$ 
$$V^{(j)}_{0,k}=\inf\{t\ge U^{(j)}_{0,k}:Y_t\not\in[\ell(j-1)+c,\ell(j+1)+c-2]_\Z\}.$$
\end{enumerate}
\end{itemize}
Finally, we let
$$V^{(j)}_k=\left\{\begin{array}{ll}
V^{(j)}_{0,k} & \mbox{if }\mathcal{C}^{(j)}_{V^{(j)}_{k-1}}=0,\\
V^{(j)}_{1,k} & \mbox{if }\mathcal{C}^{(j)}_{V^{(j)}_{k-1}}>0.
\end{array}\right.$$
{We also use notations $T^{(j)}_k, U^{(j)}_k$ in the similar manner.}
%%%%%%%%%%%%%%%%%%%%%%%%%%%%%%%%%%%%%%%%%%%
%%%Flowchart begins
%%%%%%%%%%%%%%%%%%%%%%%%%%%%%%%%%%%%%%%%%%%
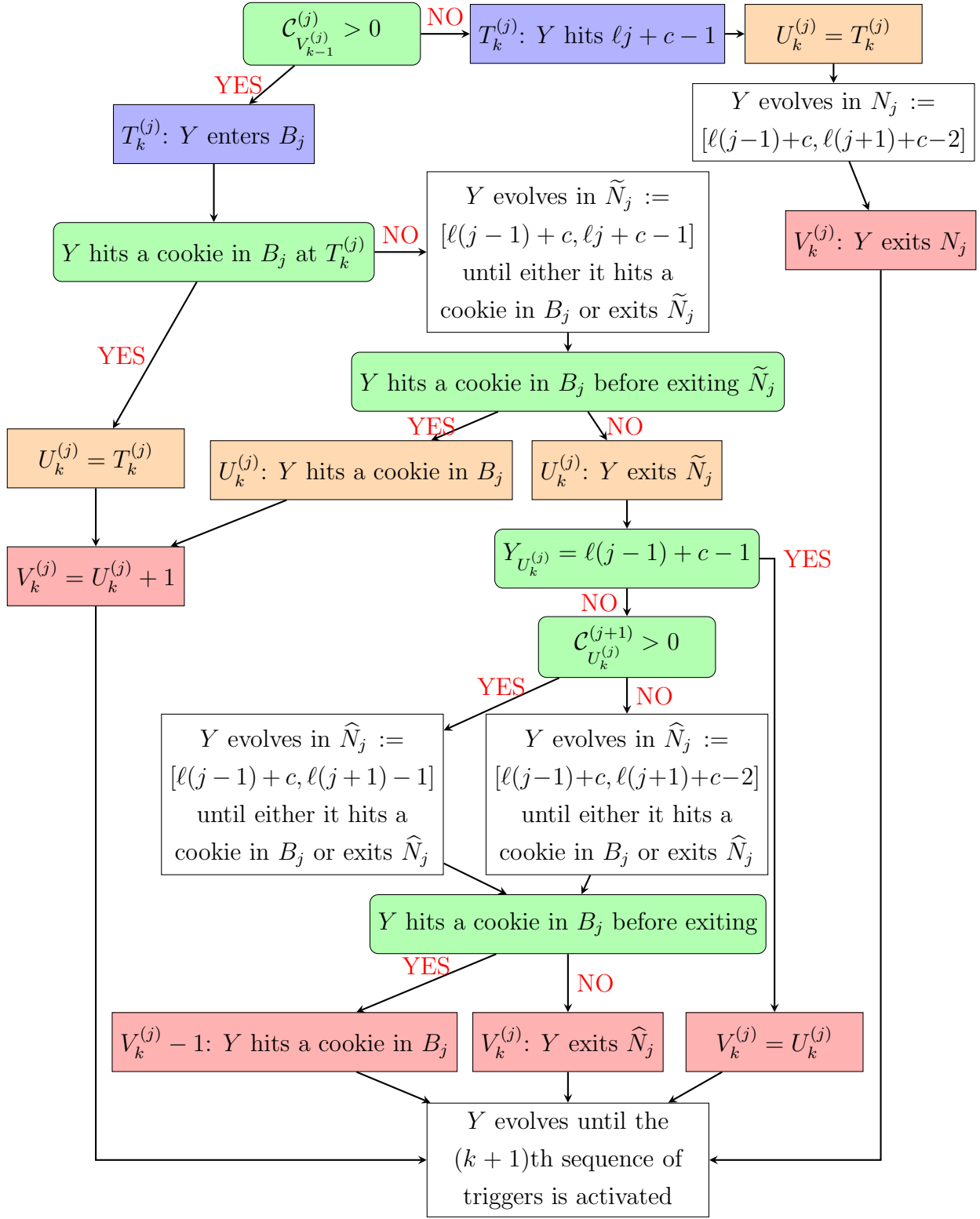
\begin{figure}
\centering
\tikzstyle{trigger1} = [rectangle, minimum width=3cm, minimum height=1cm,text centered, draw=black, fill=blue!30]
\tikzstyle{trigger2} = [rectangle, minimum width=3cm, minimum height=1cm, text centered, draw=black, fill=orange!30]
\tikzstyle{trigger3} = [rectangle, minimum width=3cm, minimum height=1cm, text centered, draw=black, fill=red!30]
\tikzstyle{comment} = [rectangle, text width=4.5cm, minimum height=1cm, text centered, draw=black]
\tikzstyle{update} = [rectangle, minimum height=1cm, text centered, draw=black, fill=yellow!30]
\tikzstyle{decision} = [rectangle, rounded corners, minimum width=3cm, minimum height=1cm, text centered, draw=black, fill=green!30]
\tikzstyle{arrow} = [thick,->,>=stealth]

\begin{tikzpicture}[node distance=2cm, scale=1, every node/.style={transform shape}]
\node (cook) [decision] {$\mathcal{C}^{(j)}_{V^{(j)}_{k-1}}>0$};

\node (enter0) [trigger1, right of=cook, xshift=2.5cm] {$T^{(j)}_k$: $Y$ hits $\ell j+c-1$};
\draw [arrow] (cook) -- node[anchor=south] {{\color{red} NO}} (enter0);

\node (U0) [trigger2, right of=enter0, xshift=2cm] {$U^{(j)}_k=T^{(j)}_k$};
\draw [arrow] (enter0) -- (U0);

\node (comU0) [comment, below of=U0, yshift=0.5cm,xshift=0cm] {$Y$ evolves in $N_j:=[\ell(j-1)+c,\ell(j+1)+c-2]$};
\draw [arrow] (U0) -- (comU0);

\node (V0) [trigger3, below of=comU0, yshift=0cm, xshift=0.8cm] {$V^{(j)}_k$: $Y$ exits $N_j$};
\draw [arrow] (comU0) -- (V0);

\node (enter1) [trigger1, below of=cook, yshift=0.3cm,xshift=-2cm] {$T^{(j)}_k$: $Y$ enters $B_j$};
\draw [arrow] (cook) -- node[anchor=east] {{\color{red} YES}} (enter1);

\node (cook1) [decision, below of=enter1] {$Y$ hits a cookie in $B_j$ at $T^{(j)}_k$};
\draw [arrow] (enter1) -- (cook1);

\node (comT10) [comment, right of=cook1, xshift=4cm] {$Y$ evolves in $\widetilde{N}_j:=[\ell(j-1)+c,\ell j+c-1]$ until either it hits a cookie in $B_j$ or exits $\widetilde{N}_j$};
\draw [arrow] (cook1) -- node[anchor=south] {{\color{red} NO}} (comT10);

\node (cook2) [decision, below of=cook1, xshift=6cm, yshift=-0.2cm] {$Y$ hits a cookie in $B_j$ before exiting ${\widetilde{N}_j}$ };
\draw [arrow] (comT10) -- (cook2);

\node (U101) [trigger2, below of=cook2, yshift=0.5cm, xshift=-3.5cm] {$U^{(j)}_k$: $Y$ hits a cookie in $B_j$};
\draw [arrow] (cook2) -- node[anchor=east] {{\color{red} YES}} (U101);

\node (U100) [trigger2, below of=cook2, yshift=0.5cm,xshift=1cm] {$U^{(j)}_k$: $Y$ exits $\widetilde{N}_j$};
\draw [arrow] (cook2) -- node[anchor=west] {{\color{red} NO}} (U100);

\node (U11) [trigger2, below of=cook1, xshift=-2cm, yshift=-1.5cm] {$U^{(j)}_k=T^{(j)}_k$};
\draw [arrow] (cook1) -- node[anchor=east] {{\color{red} YES}} (U11);

\node (VT) [trigger3, below of=U11] {$V^{(j)}_k=U^{(j)}_k+1$};
\draw [arrow] (U11) -- (VT);
\draw[arrow](U101)--(VT);

\node (EL) [decision, below of=U101, xshift=4.5cm,yshift=0.5cm] {$Y_{U^{(j)}_k}=\ell(j-1)+c-1$ };
\draw [arrow] (U100)--(EL);

\node (cook3) [decision, below of=EL,yshift=0.5cm] {$\mathcal{C}^{(j+1)}_{U^{(j)}_k}>0$};
\draw [arrow] (EL) --node[anchor=east] {{\color{red} NO}} (cook3);

\node (comU11) [comment, below of=cook3,xshift=-5.5cm,yshift=-0.5cm] {$Y$ evolves in ${\widehat{N}_j}:=[\ell(j-1)+c,\ell(j+1)-{1}]$ until either it hits a cookie in $B_j$ or exits ${\widehat{N}_j}$};
\draw [arrow] (cook3) -- node[anchor=south] {{\color{red} YES}} (comU11);

\node (comU10) [comment, below of=cook3, xshift=0cm, yshift=-0.5cm] {$Y$ evolves in ${\widehat{N}_j}:=[\ell(j-1)+c,\ell(j+1)+c-2]$ until either it hits a cookie in $B_j$ or exits {$\widehat{N}_j$}};
\draw [arrow] (cook3) -- node[anchor=west] {{\color{red} NO}} (comU10);

\node (cook4) [decision, below of=cook3, xshift=-1cm,yshift=-2.7cm] {$Y$ hits a cookie in $B_j$ before exiting};
\draw [arrow] (comU11) -- (cook4);
\draw [arrow] (comU10) -- (cook4);

\node (V10) [trigger3, below of=cook4] {$V^{(j)}_k$: $Y$ exits {$\widehat{N}_j$}};
\draw [arrow] (cook4) -- node[anchor=west] {{\color{red} NO}} (V10);

\node (V11) [trigger3, left of=V10, xshift=-2.8cm ] {$V^{(j)}_k-1$: $Y$ hits a cookie in $B_j$};
\draw [arrow] (cook4) -- node[anchor=south] {{\color{red} YES}} (V11);

\node (next10) [comment, below of=V10] {$Y$ evolves until the $(k+1)$th sequence of triggers is activated};
\draw [arrow] (V10) -- (next10);
\draw [arrow] (V0) |- (next10);
\draw [arrow] (V11) -- (next10);
\draw [arrow] (VT) |- (next10);

\node (VL) [trigger3, below of=cook4,xshift=3.5cm] {$V_k^{(j)}=U_k^{(j)}$};
\draw [arrow] (EL)-|node[anchor=west]{{\color{red} YES}} (VL);
\draw [arrow] (VL)-- (next10);

\end{tikzpicture}
\caption{The $k$-th trigger sequence associated with $B_j$, where we paint first triggers in blue, second triggers in orange, third triggers in red, and decision nodes in green.}
\label{fig:triggerseq}
\end{figure}
%%%%%%%%%%%%%%%%%%%%%%%%%%%%%%%%%%%%%%%%%%%
%%%Flowchart ends
%%%%%%%%%%%%%%%%%%%%%%%%%%%%%%%%%%%%%%%%%%%

As observed above, by virtue of $V^{(j)}_{k-1}\le\min\big(V^{(j)}_{0,k},V^{(j)}_{1,k}\big)$, $V^{(j)}_k$ is an $\mathcal F$-stopping time. As usual, $\inf\emptyset=\infty$ and if at any point in time none of the trigger events are realised, then all subsequent times are set to infinity.

Next, we define the rules that decide on the updates of the arrow system: $\mathcal{E}(j, \cdot)$ remains unchanged at time $V^{(j)}_k$ (i.e. $\mathcal{E}(j, k)=1$) if and only if
\begin{itemize}
    \item $V^{(j)}_k=\infty$, {\bf or}
    \item $V^{(j)}_k<\infty$ and $\mathcal{C}^{(j)}_{V^{(j)}_{k-1}}=0$ and $Y_{V^{(j)}_k}\ge\ell(j+1)+c -1$, {\bf or}
    \item $V^{(j)}_k<\infty$ and $\mathcal{C}^{(j)}_{V^{(j)}_{k-1}}>0$ and $\mathcal{C}^{(j+1)}_{U^{(j)}_{1,k}}=0$ and $Y_{V^{(j)}_k}\ge\ell(j+1)+c  -1$, {\bf or}
    \item $V^{(j)}_k<\infty$ and $\mathcal{C}^{(j)}_{V^{(j)}_{k-1}}>0$ and $\mathcal{C}^{(j+1)}_{U^{(j)}_{1,k}}>0$ and $Y_{V^{(j)}_k}\ge\ell(j+1)$.
\end{itemize}
{In all other cases, we set $\mathcal{E}(j, k)=-1$.}
%%%%%%%%%%%%%%%%%%%%%%%%%%%%%

\begin{prop}\label{lem:lowerb}
Fix $j \in \Z$ and $k\in\N$. Then
\begin{eqnarray}\label{eq:prob0}
& \P\Big(\mathcal{E}(j, k) =1\Big|\mathcal{F}_{T^{(j)}_{0,k}}\Big) \ge 1/2 \quad{\text{on}\quad \Big\{T^{(j)}_{0,k}<\infty,\mathcal{C}^{(j)}_{V^{(j)}_{k-1}}=0\Big\}}\quad\text{and} \\
& \label{eq:prob1}\displaystyle
    \P\Big(\mathcal{E}(j, k) =1\Big|\mathcal{F}_{T^{(j)}_{1,k}}\Big) \ge \left(1-\frac{c-1}{\ell}\right) Q(\ell+c-1)>1/2{\text{ on } \Big\{T^{(j)}_{1,k}<\infty,\mathcal{C}^{(j)}_{V^{(j)}_{k-1}}>0\Big\}}.
\end{eqnarray}
\end{prop}

\begin{proof}
Assuming that $\mathcal{C}^{(j)}_{V^{(j)}_{k-1}}=0$ and $T^{(j)}_{0,k}<\infty$, we have that $Y_{T^{(j)}_{0,k}}=\ell j+ c-1$ and the events\\
$\{{\mathcal E}(j,k)=1\}, \{{\mathcal E}(j,k)=-1\}$ are equivalent to $\big\{Y_{V^{(j)}_k}\ge\ell(j+1)+c-1\big\}$, $\big\{Y_{V^{(j)}_k}=\ell(j-1)+c-1\big\}$ respectively; and furthermore $Y_t\in [\ell(j-1)+c, \ell(j+1)+c-2]_{\Z}$ for all  $t\in [T_{0,k}^{(j)},V_{k}^{(j)}-1]_{\Z}$.
{Recall from Remark~\ref{rmk1}.ii that} $(Y_t)_{t\ge T_{0,k}^{(j)}}$ stochastically dominates a symmetric simple random walk started at the midpoint of the interval $[\ell(j-1)+c-1,\ell(j+1)+c-1]_{\Z}$. Hence, we can ascertain that, conditional on $\mathcal{F}_{T^{(j)}_{0,k}}$ and $\big\{T^{(j)}_{0,k}<\infty,\mathcal{C}^{(j)}_{V^{(j)}_{k-1}}=0\big\}$, the probability of the event $\{{\mathcal E}(j,k)=1\}=\big\{Y_{V^{(j)}_k}\ge\ell(j+1)+c-1\big\}$ is at least $1/2$.  This proves  
\eqref{eq:prob0}.

Assume now that $T_{1,k}^{(j)}<\infty$ and $\mathcal{C}^{(j)}_{V^{(j)}_{k-1}}>0$. Set $$\gamma_{j,k}=\ell \1_{\Big\{\mathcal{C}^{(j+1)}_{U_{1,k}^{(j)}}>0\Big\}}+(\ell+c-1)\1_{\Big\{\mathcal{C}^{(j+1)}_{U_{1,k}^{(j)}}=0\Big\}}.$$ Consider the process $Y$ from time $T_{1,k}^{(j)}$ (with $Y_{T_{1,k}^{(j)}}\in B_j$). We have the following cases:
\begin{itemize}
\item $Y$ hits a cookie in $B_j$ at time $U_{1,k}^{(j)}$ (and thus $V_{1,k}^{(j)}=U_{1,k}^{(j)}+1$). In this case, we have $\mathcal{E}(j,k)=1$ when $\Delta Y_{U_{1,k}^{(j)}+1}\ge \gamma_{j,k}$;
\item $Y_{U^{(j)}_{1,k}}= \ell(j-1)+c-1$ (and thus $V_{1,k}^{(j)}=U_{1,k}^{(j)}$). In this case, we always have $\mathcal{E}(j, k)=-1$;
\item $Y_{U^{(j)}_{1,k}}\ge  \ell j +c$ and $Y$ hits a cookie in $B_j$ at time $V_{1,k}^{(j)}-1$. In this case,  $\mathcal{E}(j, k)=1$ when $\Delta Y_{V_{1,k}^{(j)}}\ge \gamma_{j,k}$;
\item $Y_{U^{(j)}_{1,k}}\ge  \ell j +c$ and $Y$ does not hit any cookie in $B_j$ during $[U^{(j)}_{1,k}, V^{(j)}_{1,k}-1]_{\Z}$. In this case, we always have $\mathcal{E}(j,k)=1$ (in this case, we have $Y_{V_{1,k}^{(j)}}\ge \ell(j +1) $ if $C_{U_{1,k}^{(j)}}^{(j+1)}>0$ or $Y_{V_{1,k}^{(j)}}\ge \ell(j+1)  +c-1$ if $C_{U_{1,k}^{(j)}}^{(j+1)}=0$).
\end{itemize}
Hence, {on the event $\big\{T^{(j)}_{1,k}<\infty,\mathcal{C}^{(j)}_{V^{(j)}_{k-1}}>0\big\}$ we have}
\begin{eqnarray}\nonumber
\lefteqn{\P\Big(\mathcal{E}(j,k)=1\Big|\mathcal{F}_{T^{(j)}_{1,k}}\Big)}\\ \nonumber
& \ge  & \P\Big(\Delta\mathcal{C}^{(j)}_{U^{(j)}_{1,k}+1}=-1,\Delta Y_{U^{(j)}_k+1}\ge \gamma_{j,k}\Big|\mathcal{F}_{T^{(j)}_{1,k}}\Big)\\ \nonumber
& & + \P\Big(Y_{U^{(j)}_{1,k}}\ge \ell j + c,\Delta  \mathcal{C}^{(j)}_{V^{(j)}_{1,k}}=-1,\Delta Y_{V^{(j)}_k}\ge \gamma_{j,k}\Big|\mathcal{F}_{T^{(j)}_{1,k}}\Big)\\ \nonumber
& & + \P\Big(Y_{U^{(j)}_{1,k}}\ge \ell j + c,\Delta \mathcal{C}^{(j)}_{V^{(j)}_{1,k}}=0 \Big|\mathcal{F}_{T^{(j)}_{1,k}}\Big)\\
\label{eqr1}& = & \E\Big[\P\Big(\Delta Y_{U^{(j)}_k+1}\ge \gamma_{j,k}\Big|\mathcal{F}_{U^{(j)}_{1,k}}\Big)\1_{\Big\{\Delta\mathcal{C}^{(j)}_{U^{(j)}_{1,k}+1}=-1\Big\}}\Big|\mathcal{F}_{T^{(j)}_{1,k}}\Big]\\
\nonumber & & + \E\Big[\P\Big(\Delta Y_{V^{(j)}_k}\ge \gamma_{j,k}\Big|\mathcal{F}_{V^{(j)}_{1,k}-1}\Big)\1_{\Big\{Y_{U^{(j)}_{1,k}}\ge \ell j + c,\ \Delta\mathcal{C}^{(j)}_{V^{(j)}_{1,k}}=-1\Big\}}\Big|\mathcal{F}_{T^{(j)}_{1,k}}\Big]\\
\nonumber & & + \P\Big(Y_{U^{(j)}_{1,k}}\ge \ell j + c,\Delta \mathcal{C}^{(j)}_{V^{(j)}_{1,k}}=0 \Big|\mathcal{F}_{T^{(j)}_{1,k}}\Big)\\
\label{eqr2} & \ge & Q(\ell+c-1) \left[   \P\Big(\Delta\mathcal{C}^{(j)}_{U^{(j)}_{1,k}+1}=-1 \Big|\mathcal{F}_{T^{(j)}_{1,k}}\Big)\right. + \P\Big(Y_{U^{(j)}_{1,k}}\ge \ell j + c,\ \Delta\mathcal{C}^{(j)}_{V^{(j)}_{1,k}}=-1  \Big|\mathcal{F}_{T^{(j)}_{1,k}}\Big)\\
\nonumber & & + \left.\P\Big(Y_{U^{(j)}_{1,k}}\ge \ell j + c,\Delta \mathcal{C}^{(j)}_{V^{(j)}_{1,k}}=0 \Big|\mathcal{F}_{T^{(j)}_{1,k}}\Big) \right]\\
& = & Q(\ell+c-1) \P\Big( \Big\{Y_{U^{(j)}_{1,k}}\ge \ell j + c\Big\}\cup\Big\{\Delta\mathcal{C}^{(j)}_{U^{(j)}_{1,k}+1}=-1 \Big\}\Big|\mathcal{F}_{T^{(j)}_{1,k}} \Big),
\end{eqnarray}
where \eqref{eqr1} derives from the fact that $$\Big\{\Delta\mathcal{C}^{(j)}_{U^{(j)}_{1,k}+1}=-1\Big\}\in \mathcal{F}_{U^{(j)}_{1,k}}\quad\text{and} \quad\Big\{Y_{U^{(j)}_{1,k}}\ge \ell j + c, \Delta\mathcal{C}^{(j)}_{V^{(j)}_{1,k}}=-1\Big\}\in \mathcal{F}_{V^{(j)}_{1,k}-1};$$ \eqref{eqr2} follows from the fact that $\gamma_{j,k}\le \ell+c -1$ and the probability for $Y$ to make a jump with minimum length $\ell+c-1$ is equal to $Q(\ell+c-1)$. We notice that $\Big\{\Delta\mathcal{C}^{(j)}_{U^{(j)}_{1,k}+1}=-1\Big\}$ and $\Big\{Y_{U^{(j)}_{1,k}}\ge \ell j + c\big\}$ are disjoint and their union is equivalent to the event that $Y$ either hits a cookie in $B_j$ or crosses $\ell j+ c$ before hitting $\ell (j-1)+c-1$ (i.e. hitting $B_{j-1}$). Note also that $c-1$ is the maximum distance from $Y_{T_{1,k}^{(j)}}$ to a cookie in $B_j$ (the worst case scenario is that $Y_{T_{1,k}^{(j)}}=\ell j$ and there is only the last cookie left at $\ell j+c-1$). {Using the stochastic domination between $Y$ and} a simple symmetric random walk staring from $Y_{T_{1,k}^{(j)}}\in B_j$ {(see Remark~\ref{rmk1}.ii)},  we infer that the probability of this union event conditional on $\mathcal{F}_{T^{(j)}_{1,k}}$ and $\Big\{T^{(j)}_{1,k}<\infty,\mathcal{C}^{(j)}_{V^{(j)}_{k-1}}>0\Big\}$ is at least $1-(c-1)/\ell$. This proves \eqref{eq:prob1}. 
\end{proof}

\begin{rmk}\label{rem:cookie}We observe that on the event $\Big\{\mathcal{C}^{(j)}_{V^{(j)}_{k-1}}=0\Big\}$, $V^{(j)}_k<\infty$ if and only if $T^{(j)}_{0,k}=U^{(j)}_{0,k}<\infty$.  Similarly, on the event $\Big\{\mathcal{C}^{(j)}_{V^{(j)}_{k-1}}>0\Big\}$, $V^{(j)}_k<\infty$ if and only if $U^{(j)}_{1,k}<\infty$ if and only if $T^{(j)}_{1,k}<\infty$. 
As a result, the conditioning in \eqref{eq:prob0} and \eqref{eq:prob1} could be performed on the event $\Big\{V^{(j)}_{k-1}<\infty\Big\}$.
\end{rmk}
Having defined the stopping times $(T^{(j)}_k)_{j,k}$, we order them in an increasing sequence $(H_n)_n$. Since $Y_0=0$, we have $H_1=T^{(0)}_1$.

Note that  $(Y_{H_n})_n$ is associated with the movements between ``mega vertices" $(B_j)_j$. Keep in mind that we can not reconstruct from  $\mathcal{E}$ the transition of $(Y_{H_n})_n$ since there are possible forward jumps between non-nearest neighbour ``mega vertices". We  will later show in Proposition~\ref{prop:boh} below that the speed of the nearest-neighbour random walk generated by $\mathcal{E}$ is comparable to the speed of $(Y_{H_n})$.

Notice that the trigger-sequences, $(T,U,V)$, {do not intertwine}.  
\begin{prop}\label{3seq}
If $H_n=T^{(i)}_k$ and $H_{n+1}=T^{(j)}_m$, then
$$ T^{(i)}_k\leq U^{(i)}_k\le V^{(i)}_k\le T^{(j)}_m \quad\text{and}\quad H_n<H_{n+1}.$$
\end{prop}
The proof is straightforward from the definitions of $T^{(j)}_k, U^{(j)}_k$ and $V^{(j)}_k$.
\begin{prop}\label{pr:h}
  There exists  a constant $K$ such that
\begin{equation}\label{eq:boh1}
 {1  \le  \liminf_{n\to\infty} \frac{H_n} n  \le     \limsup_{n\to\infty} \frac{H_n} n \le  K.}
 \end{equation}
 \end{prop}
 \begin{proof}
The lower bound immediately follows from the fact that $H_{n+1}-H_n\ge1$. Suppose $H_n=T^{(i)}_k$ and $H_{n+1}=T^{(j)}_m$ for some $k$ and $m$ in $\N$. Then Proposition~\ref{3seq} justifies the following decomposition:
$$H_{n+1}-H_n =(T^{(j)}_m-V^{(i)}_k)+ (V^{(i)}_k-U^{(i)}_k)+(U^{(i)}_k-T^{(i)}_k).$$

{Using Lemma~\ref{MeanTimeExit} (see Appendix) together with \eqref{eq:constr} and Remark~\ref{rmk1}.ii}, we immediately obtain that $$\E\Big[\big(U^{(i)}_k-T^{(i)}_k\big)^2\Big|\mathcal{F}_{T^{(i)}_k}\Big]\leq \ell 2^{2\ell+1}\mbox{ and }\E\Big[\big(V^{(i)}_k-U^{(i)}_k\big)^2\Big|\mathcal{F}_{U^{(i)}_k}\Big]\leq 2\ell 2^{1+4\ell}.$$

Suppose that at time $V^{(i)}_k$, the process enters an interval $[\ell(h-1)+c,\ell h+c-1]_\Z$. We say this landing is a success if one of the following events happens:
\begin{itemize}
    \item When $\mathcal{C}^{(h)}_{V^{(i)}_k}>0$, the process $Y$ hits $\ell(h-1)+c-1$ or $B_h$ before crossing $\ell h+c$ (i.e. hitting the interval $[\ell h+c,\infty)_{\Z}$);
     \item When $\mathcal{C}^{(h)}_{V^{(i)}_k}=0$,  the process $Y$ hits $\ell(h-1)+c-1$ or $\ell h+c-1$ before crossing $\ell h+c$.
\end{itemize}
Otherwise, it is considered a failure.
It is clear that the probability of success is bounded from below by $\kappa=(q(-1))^{\ell-c}$.
Furthermore the time it takes to register the status of the landing has a second moment bounded from above by $\ell2^{2\ell+1}$.
On the event that $Y$ crosses $\ell h+c$, it enters another interval of the form $[\ell(\eta-1)+c,\ell \eta+c-1]_\Z$ with the same bounds on the probability of a successful landing and on the time to register the status of the landing. The number of landings needed to reach success is stochastically dominated by a geometric random variable with probability of success equal $\kappa$. It follows that
$$\E\Big[\big(T^{(j)}_m-V^{(i)}_k\big)^2\Big|\mathcal{F}_{V^{(i)}_k}\Big]\le \kappa^{-1} \ell2^{2\ell+1}.$$
Using Lemma~\ref{lem:average} (see Appendix), we show that for each term, the sequence of time averages has a bounded $\limsup$. The result immediately follows.
 \end{proof}

Let $X^{\mathcal E}$ be the nearest neighbour random walks constructed from arrow systems $\mathcal E$ in the manner of Section~\ref{sec:ArrowSys}.

\begin{prop}\label{prop:boh}
There exists a constant $\rho>0$ such that a.s.
\begin{equation}
    \limsup_{n\to\infty}\frac{Y_{H_{n}}}{n} \ge \rho 
    \limsup_{n\to\infty} \frac{X^{\mathcal E}_n}{n}.
\end{equation}
\end{prop}
\begin{proof}
Let $(\tau_n)_n$ be the random sequence defined recursively as follows: $\tau_0=H_0=0$, and
$$\tau_{n+1}=\inf\{m\ge \tau_n: Y_{H_m} \text{ and } Y_{H_{\tau_n}} \text{ are not in the same ``mega vertex"}\} .$$
This definition ensures that  $Y_{H_{\tau_{n+1}}}$ and $Y_{H_{\tau_{n}}}$ are not in the same ``mega vertex" for all $n$, whereas  the random sequence  $(Y_{H_{n}})_{n}$ does not necessarily satisfy this condition. 

Note that $\tau_n=\sum_{k=1}^n(\tau_{k}-\tau_{k-1})$. On the event $\{ H_n=T_{k}^{(j)}\}\subset \{ Y_{H_{n}}\in B_{j}\}$ for some $j\in \Z$ and $k\in \N$, we have
\begin{align*} & \P\Big( Y_{H_{n+1}}\in B_{j} \ | \
\mathcal{F}_{H_n}   \Big) 
 \le \P \Big(\mathcal{E}(j,k)=-1 \ |\ \mathcal{F}_{T_{k}^{(j)}}  \Big)
% & = \E\Big[\P\Big(\mathcal{E}(j,k)=-1 \ | \ \mathcal{F}_{T_{0,k}^{(j)}} \Big)\1_{\Big\{\mathcal{C}^{(j)}_{V^{(j)}_{k-1}}=0\Big\}}+\P\Big(\mathcal{E}(j,k)=-1 \ | \ \mathcal{F}_{T_{1,k}^{(j)}}  \Big)\1_{\Big\{\mathcal{C}^{(j)}_{V^{(j)}_{k-1}}>0\Big\}}  \big| \mathcal{F}_{V_{k-1}^{(j)}}\Big] 
\le 1/2,
\end{align*}
where the last inequality directly follows from Proposition~\ref{lem:lowerb}.
As a result, conditioning on ${\mathcal{F}_{H_{\tau_n}}}$, $\tau_{n+1}-\tau_{n}$ is dominated by a geometrically distributed random variable with parameter $1/2$. Hence, $\E[(\tau_{n+1}-\tau_{n})^2 | \mathcal{F}_{H_{\tau_{n}}}]\le 6$.
Using Lemma~\ref{lem:average} (see Appendix), we obtain that a.s.
\begin{equation}\label{lln.tau}
\limsup_{n\to\infty} \frac{\tau_n}{n} <\infty.
\end{equation}

We define the arrow systems $\mathcal H$ and $\mathcal K$ associated with the process $(Y_{H_{\tau_n}})_{n}$ as follows. We define the random sequence $(j_n)_n$ as  $Y_{H_{\tau_n}}\in B_{j_n}$ for each $n$. We initially set $\mathcal{H}(j,k)=\mathcal{K}(j,k)=+1$ for all $j\in \Z, k\in \N$. We will update $\mathcal{H}$ and $\mathcal{K}$ corresponding to block-transitions of $(Y_{H_{\tau_n}})_n$. Assume that for each $i\in\Z$ there are $k_i$ (resp. $h_i$) updated arrows in the arrow stack at $i$ in $\mathcal{K}$ (resp. $\mathcal{H}$). We  distinguish two cases:
\begin{itemize}
 \item If $j_{n+1}>j_n$, we keep $\mathcal{H}(j_n, h_{j_n}+1)=+1$ and keep $\mathcal{K}(i, k_{i}+1)=+1$ for all $i\in[j_n,j_{n+1}-1]_{\Z}$. 
\item If $j_{n+1}=j_n-1$, we change  $\mathcal{H}(j_n, h_{j_n}+1)=\mathcal{K}(j_n, k_{j_n}+1)= -1$ for both of the systems.
\end{itemize}
Notice that all the remaining elements of $\mathcal H$ and $\mathcal K$ which cannot be updated by block-transitions of $(Y_{H_n})_{n\ge 1}$ (due to the transience of $Y$) will be kept as $+1$.
%%%%%%%%%%%%%%%%%%%%%%%%%%%%%%%%%%%%%%%%%%%
%%%Coupling graph begins
%%%%%%%%%%%%%%%%%%%%%%%%%%%%%%%%%%%%%%%%%%%
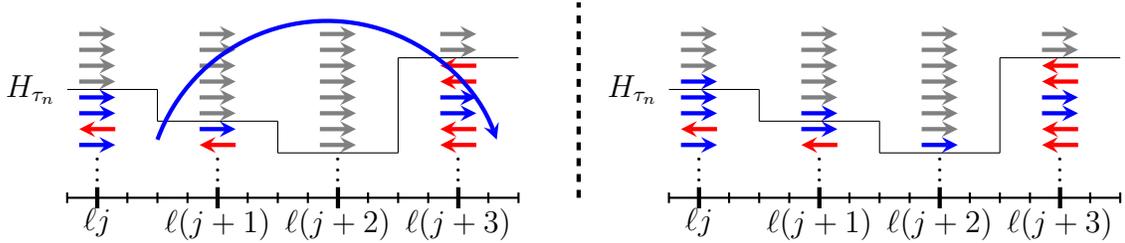
\begin{figure}[h]
\centering
\tikzstyle{cookie} = [circle, draw=black, fill=blue!30]
\tikzstyle{vertex} = [circle, draw=black, fill=red!30]
\tikzstyle{arrow} = [ultra thick,->,>=stealth]
\begin{tikzpicture}[x=0.4cm, y=1pt]
\draw[thick] (0,-12) -- (15,-12);
\foreach \x in {0,...,15}
   \draw[thick] (\x,-10) -- (\x,-14);
\foreach \x in {-1,...,2}
    \draw[ultra thick] (4*\x+5,-8) -- (4*\x+5,-16);
\node [below] at (1,-12) {$\ell j$};
\node [below] at (5,-12) {$\ell(j+1)$};
\node [below] at (9,-12) {$\ell(j+2)$};
\node [below] at (13,-12) {$\ell(j+3)$};
\node [above] at (1,-10) {$\vdots$};
\node [above] at (5,-10) {$\vdots$};
\node [above] at (9,-10) {$\vdots$};
\node [above] at (13,-10) {$\vdots$};

\draw[arrow,blue] (0.4,8) -- +(1.2,0);
\draw[arrow,red] (1.6,14) -- +(-1.2,0);
\draw[arrow,blue] (0.4,20) -- +(1.2,0);
\draw[arrow,blue] (0.4,26) -- +(1.2,0);
\foreach \y in {1,...,4}
    \draw[arrow,gray] (0.4,6*\y+26) -- +(1.2,0);

\draw[arrow,red] (5.6,8) -- +(-1.2,0);
\draw[arrow,blue] (4.4,14) -- +(1.2,0);
\foreach \y in {1,...,6}
    \draw[arrow,gray] (4.4,6*\y+14) -- +(1.2,0);

\foreach \y in {1,...,8}
    \draw[arrow,gray] (8.4,6*\y+2) -- +(1.2,0);
    
\draw[arrow,red] (13.6,8) -- +(-1.2,0);
\draw[arrow,red] (13.6,14) -- +(-1.2,0);
\draw[arrow,blue] (12.4,20) -- +(1.2,0);
\draw[arrow,blue] (12.4,26) -- +(1.2,0);
\draw[arrow,red] (13.6,32) -- +(-1.2,0);
\draw[arrow,red] (13.6,38) -- +(-1.2,0);
\foreach \y in {1,...,2}
    \draw[arrow,gray] (12.4,6*\y+38) -- +(1.2,0);

\draw (0,29) -- (3,29);
\draw (3,29) -- (3,17);
\draw (3,17) -- (7,17);
\draw (7,17) -- (7,5);
\draw (7,5) -- (11,5);
\draw (11,5) -- (11,41);
\draw (11,41) -- (15,41);
\node [left] at (0,29) {$H_{\tau_n}$};
\draw[arrow,blue] (3,10) arc (160:20:2.4cm);

%%%%%%%%%%%%%%%%%%%%%%%%%%%%%
\draw[dashed,ultra thick] (17,-12) -- (17,62);
%%%%%%%%%%%%%%%%%%%%%%%%%%%%%

\draw[thick] (20,-12) -- (35,-12);
\foreach \x in {0,...,15}
   \draw[thick] (\x+20,-10) -- (\x+20,-14);
\foreach \x in {-1,...,2}
    \draw[ultra thick] (4*\x+25,-8) -- (4*\x+25,-16);
\node [below] at (21,-12) {$\ell j$};
\node [below] at (25,-12) {$\ell(j+1)$};
\node [below] at (29,-12) {$\ell(j+2)$};
\node [below] at (33,-12) {$\ell(j+3)$};
\node [above] at (21,-10) {$\vdots$};
\node [above] at (25,-10) {$\vdots$};
\node [above] at (29,-10) {$\vdots$};
\node [above] at (33,-10) {$\vdots$};

\draw[arrow,blue] (20.4,8) -- +(1.2,0);
\draw[arrow,red] (21.6,14) -- +(-1.2,0);
\draw[arrow,blue] (20.4,20) -- +(1.2,0);
\draw[arrow,blue] (20.4,26) -- +(1.2,0);
\draw[arrow,blue] (20.4,32) -- +(1.2,0);
\foreach \y in {1,...,3}
    \draw[arrow,gray] (20.4,6*\y+32) -- +(1.2,0);

\draw[arrow,red] (25.6,8) -- +(-1.2,0);
\draw[arrow,blue] (24.4,14) -- +(1.2,0);
\draw[arrow,blue] (24.4,20) -- +(1.2,0);
\foreach \y in {1,...,5}
    \draw[arrow,gray] (24.4,6*\y+20) -- +(1.2,0);

\draw[arrow,blue] (28.4,8) -- +(1.2,0);
\foreach \y in {1,...,7}
    \draw[arrow,gray] (28.4,6*\y+8) -- +(1.2,0);

\draw[arrow,red] (33.6,8) -- +(-1.2,0);
\draw[arrow,red] (33.6,14) -- +(-1.2,0);
\draw[arrow,blue] (32.4,20) -- +(1.2,0);
\draw[arrow,blue] (32.4,26) -- +(1.2,0);
\draw[arrow,red] (33.6,32) -- +(-1.2,0);
\draw[arrow,red] (33.6,38) -- +(-1.2,0);
\draw[arrow,gray] (32.4,44) -- +(1.2,0);
\draw[arrow,gray] (32.4,50) -- +(1.2,0);

\draw (20,29) -- (23,29);
\draw (23,29) -- (23,17);
\draw (23,17) -- (27,17);
\draw (27,17) -- (27,5);
\draw (27,5) -- (31,5);
\draw (31,5) -- (31,41);
\draw (31,41) -- (35,41);
\node [left] at (20,29) {$H_{\tau_n}$};
\end{tikzpicture}
\caption{
The arrow system $\mathcal K$: before (left) and after (right) the forward block-jump of $(Y_{H_{\tau_n}})_n$ from $B_{j}$ to $B_{j+3}$.  For illustrative purposes, we put the arrow stack $\mathcal K(j,.)$ at   $\ell j$. We paint right-pointing arrows in green, left-pointing arrows in red and not updated arrows in gray.
}
\label{fig:couplingA}
\end{figure}
%%%%%%%%%%%%%%%%%%%%%%
% Coupling graph ends
%%%%%%%%%%%%%%%%%%%%%%

Let $X^{\mathcal H}, X^{\mathcal K}$ be  the nearest neighbour random walks constructed from arrow systems $\mathcal H$ and $\mathcal K$ respectively. By the above construction, $\mathcal K$ can be obtained from $\mathcal{H}$ by inserting more +1 elements associated with long forward jumps of $(Y_{H_n})_n$. Hence, we have $\mathcal K \succeq \mathcal H$. By virtue of Theorem~\ref{th:order},  
\begin{equation}\label{ineq1}
\limsup_{n\to\infty} \frac{X^{\mathcal K}_n}{n}\ge \limsup_{n\to\infty} \frac{X^{\mathcal H}_n}{n}.
\end{equation}
Let $\mathcal L^+_{\mathcal H}(n,i)$ (resp. $\mathcal L^+_{\mathcal E}(n,i)$) be the number of $+1$ in the set $\{\mathcal H(i, k): k\in[n]\}$ (resp. $\{\mathcal K(i, k): k\in[n]\}$). Note that $\mathcal L^+_{\mathcal E}(n,i) \le \mathcal L^+_{\mathcal H}(n,i)$ for each $n$ and $i$. Hence, we have $\mathcal H \succeq \mathcal E$. Applying Theorem~\ref{th:order} again, we have
\begin{equation}\label{ineq2}\limsup_{n\to\infty} \frac{X^{\mathcal H}_n}{n}\ge  \limsup_{n\to\infty} \frac{X^{\mathcal E}_n}{n}.\end{equation}

Let $(\sigma_n)_n$ be the random times defined by $\sigma_0=0$ and  $\sigma_n=|j_{n}-j_{n-1}|$ for all $n\ge 1$. Note that $\ell X^{\mathcal K}_{\sigma_n} \le Y_{H_{\tau_n}} \le \ell X^{\mathcal K}_{\sigma_n} +c-1 $.
Hence $$\limsup_{n\to\infty}\frac{Y_{H_{n}}}{n} \ge \limsup_{n\to\infty}\frac{Y_{H_{\tau_n}}}{\tau_n}=\ell \limsup_{n\to\infty}\frac{X^{\mathcal K}_{\sigma_n}}{\tau_n}.$$
On the other hand, we have either $X^{\mathcal K}_{\sigma_{n+1}}=X^{\mathcal K}_{\sigma_{n}}-1$ or   $X^{\mathcal K}_{\sigma_n} \le X^{\mathcal K}_{k}\le X^{\mathcal K}_{\sigma_{n+1}}$ for any $k\in[\sigma_n,\sigma_{n+1}]$. It immediately follows that, 
$\limsup_{n\to\infty}{X^{\mathcal K}_{\sigma_n}}/{\sigma_n} = \limsup_{n\to\infty} {X^{\mathcal K}_{n}}/{n}.$ Hence,
\begin{equation}\label{ineq3}  \limsup_{n\to\infty}\Big(\frac{\tau_n}{n}\Big)  \limsup_{n\to\infty}\frac{Y_{H_{n}}}{n} \ge \ell \limsup_{n\to\infty}\Big( \frac{\tau_n}{n}\cdot\frac{n}{\sigma_n}\cdot \frac{X^{\mathcal K}_{\sigma_n}}{\tau_n}\Big)  =\ell \limsup_{n\to\infty}\frac{X^{\mathcal K}_{n}}{n} 
\end{equation}
(Here we use the fact that for any non-negative sequences $(a_n)_n, (b_n)_n$, the inequality $$ \limsup_{n\to\infty}(a_n b_n)\le \limsup_{n\to\infty}(a_n) \limsup_{n\to\infty}(b_n)$$ holds whenever $\limsup_{n\to\infty}a_n< \infty$). Combining  \eqref{lln.tau}-\eqref{ineq1}-\eqref{ineq2}-\eqref{ineq3}, we conclude the lemma. 
\end{proof}

\begin{proof}[Proof of Theorem~\ref{th:mainth}]

Let
$\eps' = 1-\left(1- ({c-1})/{\ell}\right)Q(\ell+c-1),$
where $c$ and $\ell$ satisfy the assumptions in Theorem~\ref{th:mainth}. In particular \eqref{eq:epsell} implies that  $$ \delta^* :=  c(1-2\eps') >2.$$
We first construct an arrow system $\mathcal M$ which generates a nearest neighbour $c$-cookie random walk $X^{{\mathcal{M}}}$ such that its cookie environment has expected total drift $\delta^*>2$
and 
\begin{equation}
\label{eq:dom} \P\Big(
\limsup_{n\to\infty}\frac{X^{{\mathcal{M}}}_n}{n}>0 \Big) \le \P\Big(\limsup_{n\to\infty}\frac{X^{{\mathcal{E}}}_n}{n}>0\Big).
\end{equation}

Assume $H_{n-1}=T_m^{(i)}<H_n=T_k^{(j)}<\infty$ for some $k,m\in \N$ and $i, j\in \Z$. Note that $V_{k-1}^{(j)}\le V_m^{(i)} < T_{k}^{(j)}\le V_{k}^{(j)}$. By virtue of Proposition~\ref{lem:lowerb} and Remark~\ref{rem:cookie}, we have 
\begin{align*}
   \P \Big(\mathcal{E}(j,k)= 1 \ |\ \mathcal{F}_{V_m^{(i)}}  \Big)
& \ge \frac{1}{2}\1_{\Big\{ \mathcal{C}^{(j)}_{V^{(j)}_{k-1}}=0\Big\}}+   \Big(1-\frac{c-1}{\ell}\Big)Q(\ell+c-1)\1_{ \Big\{ \mathcal{C}^{(j)}_{V^{(j)}_{k-1}}>0\Big\}}.
\end{align*}
Also note that if $Y$ hits vertices in $B_j$ at least $c$ times, the inequality $\mathcal{C}^{(j)}_{V_{k-1}^{(j)}}>0$ must hold for at least $c$ different values of $k$. Therefore,
$\P \left(\mathcal{E}(j,k)= 1 \ |\ \mathcal{F}_{V_{m}^{(i)}}\right)$
 is at least $(1-(c-1)/\ell)Q(\ell+c-1)$ for $1\le k\le c$ and at least $1/2$ for $k\ge c+1$. By reason of Lemma~\ref{lem.dom} (see Appendix), there exist two arrow systems $\mathcal M$ and $\widehat{\mathcal{E}}$  such that 
 \begin{itemize}
 \item $\widehat{\mathcal{E}}$ has the same distribution as ${\mathcal{E}}$;
 \item $\mathcal M (j,k)\le \widehat{\mathcal{E}} (j,k)$ a.s. for all $j\in \Z$ and $k\in \N$,
\item $({\mathcal{M}}(j,k))_{j\in \Z, k\in N}$ are independent random variables such that $$\P(\mathcal{M}(j,k)=1)=(1-(c-1)/\ell)Q(\ell+c-1)\text{
 for }1\le k\le c\text{ and }\P({\mathcal{M}}(j,k)=1)=1/2\text{ for }k\ge c+1.$$ 
 \end{itemize}
 As a result, $\widehat{\mathcal{E}}\succeq \mathcal{M}$ (a.s.) and by virtue of Theorem~\ref{th:order}, we thus obtain \eqref{eq:dom}. 
 
 We also notice that the nearest-neighbour random walk $X^{{\mathcal M}}$ generated by the arrow system ${\mathcal{M}}$ is a $c$-cookie random walk and its cookie environment has expected total drift $\delta^*>2$.
Using Theorem 1.1. in \cite{BasS1}, we have that a.s. $\lim_n X^{{\mathcal{M}}}_n/n$ exists and is positive. Hence, this fact and \eqref{eq:dom} imply that a.s. $\limsup_n X^{{\mathcal{E}}}_n/n>0$.

On the other hand, Proposition~\ref{prop:boh} implies that  a.s.
$$0<\limsup_{n\to \infty} \rho \frac{X_n^{{\mathcal{E}}}}{n} \le  \limsup_{n \ti} \frac{Y_{H_n}}{n}= \limsup_{n \ti} \frac{Y_{H_n}}{H_n} \frac{H_n}n \le \limsup_{n \ti} \frac{Y_{H_n}}{H_n} \limsup_{n \ti} \frac{H_n}n.$$
Using Proposition~\ref{pr:h} we have that $\limsup H_n/n$ is bounded by a constant. Hence, $\limsup Y_n/n >0$ (a.s.).
By Proposition~\ref{prop:exs}, the speed for $Y$, $\lim Y_n/n$, exists. This ends the proof.
\end{proof}

\section*{Appendix}

\begin{lem}\label{MeanTimeExit}
Let $\mathcal{G}$ be a $\sigma$-field, $$X_n=X_0+\sum_{k=1}^n\phi_k(X_{[k-1]},\zeta_k),$$ where $X_0$ is $\mathcal{G}$-measurable, $\zeta_1,\zeta_2,\ldots$ are independent uniform $[0,1]$ random variables independent of $\mathcal{G}$ and $\phi_k(x_{[k-1]},u)\ge 1_{u\le1/2}-1_{u>1/2}$, and $x_{[k]}=(x_0,\ldots,x_k)$ denotes the path of the sequence $(x_n)_n$ up to instant $k$. Then the mean time for the process $X$ to exit a bounded interval is finite. In fact, for $a,b\in\Z$, $a<b$, if $T=\inf\{n\ge0:X_n\le a\mbox{ or }X_n\ge b\}$, then $$\E[T^2|\mathcal{G}]\leq (b-a)2^{2(b-a)+1}.$$
\end{lem}
\begin{proof}
$T\le G(b-a)$ where $G=\inf\{n\ge1:\zeta_{1+(n-1)(b-a)}<1/2,\ldots,\zeta_{n(b-a)}<1/2\}$. The result follows from the fact that $G$ is geometric $(1/2)^{b-a}$ and independent of $\mathcal{G}$.
\end{proof}

\begin{lem}\label{lem:average}
Let $(\xi_n)_n$ be an arbitrary sequence of random variables adapted to a filtration $(\mathcal{F}_n)_n$ and let
$X_n=\sum_{k=1}^n\xi_k.$
Suppose there exists a positive constant $K$ such that, for each $n$,
\begin{equation}\label{eq:condvar}
\mathbb{E}[\xi_n^2|\mathcal{F}_{n-1}]\le K^2\quad (a.s.).
\end{equation}
Then
$$\limsup_{n\to\infty}\frac{X_n}n\le K\quad (a.s.).$$
\end{lem}
\begin{proof}
Let $M_n=\sum_{k=1}^n\left({\xi_{k}}-\mathbb{E}[\xi_k|\mathcal{F}_{k-1}]\right) = M_{n-1}+\xi_n-\mathbb{E}[\xi_n|\mathcal{F}_{n-1}]$, $M_0=0$. Then $(M_n)_n$ is a square integrable martingale with quadratic variation
$$A_n = \sum_{k=1}^n\mathbb{E}[(M_k-M_{k-1})^2|\mathcal{F}_{k-1}] = \sum_{k=1}^n\text{var}\big(\xi_k|\mathcal{F}_{k-1}\big)$$
and $\lim_{n\to\infty}M_n$ exists almost surely on $\{A_\infty<\infty\}$. It follows that
$$\lim_{n\to\infty}\frac{M_n}n = 0\quad \text{on} \quad\{A_\infty<\infty\}\quad\text{(a.s.)}.$$

Furthermore, from \eqref{eq:condvar} we deduce that $A_n\le nK^2$ and by application of the Law of Large Numbers for Martingales (see for example 12.14 in \cite{W1991}), we obtain that $\lim_{n\to\infty}M_n/A_n = 0$ almost surely on $\{A_\infty=\infty\}$. We immediately deduce that
$$\lim_{n\to\infty}\frac{M_n}n = \lim_{n\to\infty} \frac{M_n}{A_n}\cdot \frac{A_n}n = 0\quad\text{on}\quad\{A_\infty=\infty\},\quad\text{(a.s.)}$$  and that
$$\limsup_{n\to\infty}\frac{X_n}n = \limsup_{n\to\infty}\frac1n\Big(M_n+\sum_{k=1}^n\mathbb{E}[\xi_k|\mathcal{F}_{k-1}]\Big)
\le \limsup_{n\to\infty}\frac1n\sum_{k=1}^n\mathbb{E}[|\xi_k||\mathcal{F}_{k-1}] \le K\ {\rm (a.s.)}.$$
\end{proof}

%For two sequence $a=(a_n)_n, b=(b_n)_n$, we say $b$ is dominated by $a$ and denote $a \succeq b $ if $a_{n}\ge b_n$ for all $n$.  For two sequences of random variables $X=(X_n)_n$, $Y=(Y_n)_n$, we say $Y$ is stochastically dominated by $X$ and denote $X \succeq_{st} Y$ if $$\P(X_{n}\le u\ | \ (X_1,X_2,...X_{n-1})= x) \le \P(Y_{n}\le u\ |\  (Y_1,Y_2,...Y_{n-1})= y)$$  for all $u\in \R$ and $x\succeq y$.

The following result is a direct consequence of Strassen's theorem on stochastic dominance for sequences of random variables (see Theorem 5.8, Chapter IV, p. 134 in \cite{Lindvall}).

\begin{lem}\label{lem.dom}
Assume that $(X_n)_{n\in \N}$ is a sequence of random variables such that $\P(X_1\le x) \le G_1(x)$ and
$$\P(X_n\le x \ |\ \mathcal{F}_{n-1})\le G_n(x) \text{ for all } x\in \R \text{ and } n\ge 2$$
where $\mathcal{F}_n=\sigma(X_1,\dots, X_n)$ and $(G_n)_{n\in \N}$ is a sequence of cumulative distribution functions. Then there exist sequences of r.v. $(\widehat{X}_n)_{n\in \N}$ and $(Y_n)_{n\in \N}$ such that \begin{itemize}
    \item  $(\widehat{X}_n)_n$ has the same distribution as $({X}_n)_{n\in \N}$, \item $(Y_n)_{n\in \N}$ are independent and $G_n$ is the cumulative distribution function of $Y_n$,
    \item $\widehat{X}_n\ge Y_n$ a.s. for all $n\in \N$.
\end{itemize}
\end{lem}

\section*{Acknowledgement} {
The authors would like to thank the anonymous referees for their thorough reading and their constructive suggestions which improved the manuscript. The authors were supported by ARC grant DP180100613. Moreover, A.C. was supported by the Australian Research Council Centre of Excellence for Mathematical and Statistical Frontiers 
(CE140100049).}

\end{document}